\documentclass[a4paper]{amsart}

\usepackage{amssymb}              
\usepackage{lmodern,bm}              
\usepackage{longtable}
\usepackage{tikz}
\usepackage[all]{xy}
\usepackage{dsfont}
\usepackage{pdflscape}
\usepackage{bigdelim}
\usepackage{multirow}
\usepackage{cancel}
\usepackage{array}

\CompileMatrices

\newtheorem{theorem}{Theorem}[section]

\newtheorem{lemma}[theorem]{Lemma}
\newtheorem{proposition}[theorem]{Proposition}
\newtheorem{corollary}[theorem]{Corollary}

\theoremstyle{definition}

\newtheorem{definition}[theorem]{Definition}
\newtheorem{construction}[theorem]{Construction}

\newtheorem{remark}[theorem]{Remark}

\numberwithin{equation}{theorem}

\newcommand{\git}{\mathbin{
  \mathchoice{/\mkern-6mu/}
    {/\mkern-6mu/}
    {/\mkern-5mu/}
    {/\mkern-5mu/}}}

\def\vector2#1#2{\left(\begin{array}{c} #1 \\ #2 \end{array}\right)}

\def\CC{{\mathbb C}}

\def\TT{{\mathbb T}}
\def\ZZ{{\mathbb Z}}

\def\quot{/\!\!/}

\def\im{{\rm im}}

\def\bangle#1{{\langle #1 \rangle}}

\def\Spec{{\rm Spec}}

\def\lcm{{\rm lcm}}
\def\im{{\rm im}}

\title[On iteration of Cox rings]%
{On iteration of Cox rings}

\author[J.~Hausen and M.~Wrobel]{J\"urgen~Hausen and Milena Wrobel}

\address{Mathematisches Institut, Universit\"at T\"ubingen,
Auf der Morgenstelle 10, 72076 T\"ubingen, Germany}
\email{juergen.hausen@uni-tuebingen.de}

\address{Mathematisches Institut, Universit\"at T\"ubingen,
Auf der Morgenstelle 10, 72076 T\"ubingen, Germany}
\email{milena.wrobel@math.uni-tuebingen.de}


\begin{document}

\begin{abstract}
We characterize all varieties with a torus action 
of complexity one that admit iteration of Cox rings. 
\end{abstract}

\maketitle

\section{Introduction}

We consider normal algebraic varieties $X$
defined over the field~$\CC$ of complex numbers.
If $X$ has finitely generated divisor class group 
$K$ and only constant invertible global regular 
functions, then one defines the $K$-graded Cox 
ring $R_1$ of~$X$ as follows, see~\cite{ArDeHaLa} 
for details:
$$ 
R_1 \ = \ \bigoplus_{K} \Gamma(X,\mathcal{O}(D)).
$$
If the Cox ring $R_1$ is a finitely generated 
$\CC$-algebra, then one has the total coordinate
space $X_1 := \Spec \, R_1$.
We say that $X$ admits \emph{iteration of Cox rings}
if there is a chain
$$ 
\xymatrix{
X_p
\ar[r]^{\quot H_{p-1}}
&
X_{p-1}
\ar[r]^{\quot H_{p-2}}
&
\quad
\ldots
\quad
\ar[r]^{\quot H_{2}}
&
X_{2}
\ar[r]^{\quot H_{1}}
&
X_1
}
$$
dominated by a factorial variety $X_p$ 
where in each step, $X_{i+1}$ is the total 
coordinate space of $X_i$ and 
$H_i = \Spec \, \CC[K_i]$ the 
characteristic quasitorus of $X_i$, having 
the divisor class group $K_i$ of $X_i$ as 
its character group.
Note that if the divisor class group 
$K$ of $X$ is torsion free, then $R_1$ 
is a unique factorization domain and 
iteration of Cox rings is trivially
possible.
As soon as $K$ has torsion, it may 
happen that during the iteration 
process a total coordinate space with 
non-finitely generated divisor class group
pops up and thus there is no chain
of total coordinate spaces as above, 
see~\cite[Rem.~5.12]{ArBrHaWr}.

In~\cite{ArBrHaWr} we studied normal, rational,
$\TT$-varieties $X$ of complexity one,
where the latter means that $X$ comes with an
effective torus action $\TT \times X \to X$ 
such that $\dim(\TT) = \dim(X)-1$ holds.
We showed that for affine $X$ with 
$\Gamma(X, \mathcal{O})^{\TT}= \CC$
and at most log terminal singularities, 
the iteration of Cox rings is possible.   
In the present article, we characterize
all varieties $X$ with a torus action of 
complexity one that admit iteration of 
Cox rings.

First consider the case 
$\Gamma(X, \mathcal{O})^{\TT}= \CC$.
In order to have finitely generated divisor
class group, $X$ must be rational and
then the Cox ring of $X$ is of the form
$R = \CC[T_{ij},S_k] / I$, with a 
polynomial ring $\CC[T_{ij},S_k]$ in
variables $T_{ij}$ and $S_k$ modulo the
ideal $I$ generated by the trinomial relations
$$ 
T_0^{l_0} + T_1^{l_1} +T_2^{l_2},
\quad
\theta_1 T_1^{l_1} + T_2^{l_2} +T_3^{l_3},
\quad \ldots, \quad
\theta_{r-2} T_{r-2}^{l_{r-2}} + T_{r-1}^{l_{r-1}} +T_r^{l_r},
$$
with $T_i^{l_i} = T_{i1}^{l_{i1}} \cdots T_{in_i}^{l_{in_i}}$.
For each exponent vector $l_i$ set 
$\mathfrak{l}_i := \gcd(l_{i1},\ldots,l_{in_i})$.
We say that $R$ is \emph{hyperplatonic} if 
$\mathfrak{l}_0^{-1} + \ldots + \mathfrak{l}_r^{-1} > r-1$ 
holds.
After reordering 
$\mathfrak{l}_0, \ldots, \mathfrak{l}_r$
decreasingly,
the latter condition precisely means that 
$\mathfrak{l}_i = 1$ holds for all $i \ge 3$
and 
$(\mathfrak{l}_0,\mathfrak{l}_1,\mathfrak{l}_2)$
is a platonic triple, i.e., a triple of the form
$$ 
(5,3,2),
\quad
(4,3,2),
\quad
(3,3,2),
\quad
(x,2,2),
\quad
(x,y,1),
\quad 
x,y \in \ZZ_{\ge 1}.
$$ 

\goodbreak

\begin{theorem}
\label{theo:iteration}
Let $X$ be a normal $\TT$-variety of 
complexity one with
$\Gamma(X, \mathcal{O})^{\TT} = \CC$.
Then the following statements are equivalent.
\begin{enumerate}
\item
The variety $X$ admits iteration of Cox rings.
\item
The variety $X$ is rational with hyperplatonic
Cox ring.
\end{enumerate}
\end{theorem}

We turn to the case 
$\Gamma(X, \mathcal{O})^{\TT} \ne \CC$.
Here, $\mathcal{O}(X)^* = \CC^*$ and 
finite generation of the divisor class group 
of $X$ force 
$\Gamma(X,\mathcal{O})^{\TT} = \CC[T]$.
In this situation, we obtain the following 
simple characterization.

\begin{theorem}
\label{theo:iterationType1}
Let $X$ be a normal $\TT$-variety of complexity one
with $\Gamma(X, \mathcal{O})^{\TT} \ne \CC$.
Then $X$ admits Cox ring iteration if and only if $X$ and 
its total coordinate space are rational.
Moreover, if the latter holds, then the Cox ring iteration
stops after at most one step.
\end{theorem}

As a consequence of the two theorems above, 
we obtain the following structural result, 
generalizing~\cite[Thm.~3]{ArBrHaWr}, 
but using analogous ideas for the proof. 

\begin{corollary}
\label{cor:quot}
Let $X$ be a normal, rational variety 
with a torus action of complexity one 
admitting iteration of Cox rings. 
Then $X$ is a quotient $X = X' \git G$
of a factorial affine variety $X':= \Spec(R')$,
where $R'$ is a factorial ring
and $G$ is a solvable reductive group.
\end{corollary}

On our way of proving Theorem~\ref{theo:iteration},
we give in Proposition \ref{prop::isotropy} 
an explicit description of the Cox ring 
of a variety $\Spec \, R $ for a hyperplatonic 
ring $R$. 
This allows to describe the possible 
Cox ring iteration chains more in detail.
After reordering the numbers
$\mathfrak{l}_0, \ldots, \mathfrak{l}_r$
associated with $R$ decreasingly,
we call $(\mathfrak{l}_0,\mathfrak{l}_1,\mathfrak{l}_2)$
the~\emph{basic platonic triple} of $R$.

\begin{corollary}
\label{cor:quan}
The possible sequences of basic platonic triples 
arising from Cox ring iterations of normal, rational 
varieties with a torus action of complexity one 
and hyperplatonic Cox ring are the following: 
\begin{enumerate}
\item 
$(1,1,1) \rightarrow (2,2,2) \rightarrow (3,3,2) \rightarrow (4,3,2)$,
\item 
$(1,1,1) \rightarrow (x,x,1) \rightarrow (2x,2,2)$,
\item 
$(1,1,1) \rightarrow (x,x,1) \rightarrow (x,2,2)$,
\item 
$(\mathfrak{l}_{01}^{-1} \mathfrak{l}_0,
\mathfrak{l}_{01}^{-1} \mathfrak{l}_1,1)
\rightarrow
(\mathfrak{l}_0, \mathfrak{l}_1,1)$,
where $\mathfrak{l}_{01} := \gcd(\mathfrak{l}_0, \mathfrak{l}_1) > 1$.
\end{enumerate}
\end{corollary}

\tableofcontents

\section{Proof of Theorem~\ref{theo:iteration}}

We will work in the notation of~\cite{HaHe,HaWr}, 
where the Cox ring of a rational $T$-variety of 
complexity one is encoded by a pair of defining 
matrices.
Let us briefly recall the precise definitions
we need from~\cite{HaWr};
note that the setting will be slightly more 
flexible than the informal one given in the 
introduction.

\begin{construction}
\label{constr:RAP0Type2}
Fix integers $r,n > 0$, $m \ge 0$ and a partition 
$n = n_0+\ldots+n_r$.
For every $i = 0, \ldots, r$, fix a tuple
$l_{i} \in \ZZ_{> 0}^{n_{i}}$ and define a monomial
$$
T_{i}^{l_{i}}
\ := \
T_{i1}^{l_{i1}} \cdots T_{in_{i}}^{l_{in_{i}}}
\ \in \
\CC[T_{ij},S_{k}; \ 0 \le i \le r, \ 1 \le j \le n_{i}, \ 1 \le k \le m].
$$
We will also write $\CC[T_{ij},S_{k}]$ for the 
above polynomial ring. 
Let $A:= (a_0, \ldots, a_r)$ be a $2 \times (r+1)$~matrix with pairwise 
linearly independent columns $a_i \in \CC^2$. 
For every $i = 0, \dots, r-2$ we define
$$
g_{i}
\ :=  \
\det
\left[
\begin{array}{lll}
T_i^{l_i} & T_{i+1}^{l_{i+1}} & T_{i+2}^{l_{i+2}}
\\
a_i & a_{i+1}& a_{i+2}
\end{array}
\right]
\ \in \
\CC[T_{ij},S_{k}].
$$
We build up an $r \times (n+m)$~matrix 
from the exponent vectors $l_0, \ldots, l_r$ of these 
polynomials:
$$
P_{0}
\ := \
\left[
\begin{array}{ccccccc}
-l_{0} & l_{1} &  & 0 & 0  &  \ldots & 0
\\
\vdots & \vdots & \ddots & \vdots & \vdots &  & \vdots
\\
-l_{0} & 0 &  & l_{r} & 0  &  \ldots & 0
\end{array}
\right].
$$
Denote by $P_0^*$ the transpose of $P_0$ and consider 
the projection
$$
Q \colon \ZZ^{n+m} 
\ \to \ 
K_{0} 
\ := \ 
\ZZ^{n+m}/\mathrm{im}(P_{0}^{*}).
$$
Denote by $e_{ij},e_{k} \in \ZZ^{n+m}$ the canonical
basis vectors corresponding to the variables 
$T_{ij}$, $S_{k}$.
Define a $K_0$-grading on $\CC[T_{ij},S_{k}]$ 
by setting
$$
\deg(T_{ij}) \ := \ Q(e_{ij}) \ \in \ K_{0},
\qquad
\deg(S_{k}) \ := \ Q(e_{k}) \ \in \ K_{0}.
$$
This is the finest possible grading of
$\CC[T_{ij},S_{k}]$ leaving the variables 
and the $g_i$ homogeneous.
In particular, we have  a $K_{0}$-graded 
factor algebra
$$
R(A,P_{0})
\ := \
\CC[T_{ij},S_{k}] / \bangle{g_{0}, \dots, g_{r-2}}.
$$
\end{construction}

By the results of~\cite{HaHe,HaWr} the
rings $R(A,P_0)$ are normal complete
intersections, 
admit only constant homogeneous units
and we have unique factorization in the 
multiplicative monoid of $K_0$-homogeneous
elements of $R(A,P_0)$.
Moreover, suitably downgrading the rings 
$R(A,P_0)$ leads to the Cox rings of the 
normal rational $T$-varieties $X$ of 
complexity one with 
$\Gamma(X,\mathcal{O})^\TT = \CC$,
see~\cite{HaSu,HaHe,HaWr}.

In order to iterate a Cox ring $R(A,P_0)$,
it is necessary that $\Spec \, R(A,P_0)$
has finitely generated divisor class group.
The latter turns out to be equivalent to 
rationality of $\Spec \, R(A,P_0)$.
From~\cite[Cor.~5.8]{ArBrHaWr}, we infer 
the following rationality criterion.

\begin{remark}
\label{rem:cor}
Let $R(A,P_0)$ be as in Construction \ref{constr:RAP0Type2}
and set $\mathfrak{l}_i := \gcd(l_{i1}, \ldots, l_{in_i})$.
Then $\Spec \, R(A,P_0)$ is rational if 
and only if one of the following conditions
holds:
\begin{enumerate}
\item
We have $\gcd(\mathfrak{l}_i,\mathfrak{l}_j) = 1$
for all $0 \le i < j \le r$, in other words, $R(A,P_0)$ 
is factorial.
\item
There are $0 \le i < j \le r$ with 
$\gcd(\mathfrak{l}_i,\mathfrak{l}_j) > 1$
and $\gcd(\mathfrak{l}_u,\mathfrak{l}_v) = 1$ 
whenever $v \not\in \{i,j\}$.
\item
There are $0 \le i < j < k \le r$
with 
$
\gcd(\mathfrak{l}_i,\mathfrak{l}_j) = 
\gcd(\mathfrak{l}_i,\mathfrak{l}_k) =
\gcd(\mathfrak{l}_j,\mathfrak{l}_k) =
2
$ 
and $\gcd(\mathfrak{l}_u,\mathfrak{l}_v) = 1$ 
whenever $v \not\in \{i,j,k\}$.
\end{enumerate} 
\end{remark}

\goodbreak

\begin{definition}
Let $R(A,P_0)$ be as in 
Construction \ref{constr:RAP0Type2}
such that $\Spec \, R(A,P_0)$ is rational.
We say that~$P_0$ is \emph{$\gcd$-ordered} 
if it satisfies the following two properties
\begin{enumerate}
\item
$\gcd(\mathfrak{l}_i, \mathfrak{l}_j) =1$ 
for all $i = 0, \ldots, r$ and $j = 3, \ldots, r$,
\item
$
\gcd(\mathfrak{l}_1, \mathfrak{l}_2) 
= 
\gcd(\mathfrak{l}_0,\mathfrak{l}_1, \mathfrak{l}_2 )$.
\end{enumerate}
\end{definition}

Observe that if $\Spec \, R(A,P_0)$ is rational,
then one can always achieve that $P_0$ is 
$\gcd$-ordered by suitably reordering $l_0, \ldots, l_r$.
This does not affect the $K_0$-graded algebra 
$R(A,P_0)$ up to isomorphy.

\begin{lemma}
\label{lemma::Ptors}
Let $R(A,P_0)$ be as in 
Construction \ref{constr:RAP0Type2} 
such that $\Spec \, R(A,P_0)$ 
is rational and $P_0$ is $\gcd$-ordered.
Then, with $K_0 = \ZZ^{n+m}/\im(P_0^*)$,
the kernel of $ \ZZ^{n+m}\to K_0/K_0^{\mathrm{tors}}$ 
is generated by the rows of 
$$
P_1
\ := \
\left[
\begin{array}{ccccccccc}
\frac{-1}{\gcd(\mathfrak{l}_0,\mathfrak{l}_1)} l_0
& 
\frac{1}{\gcd(\mathfrak{l}_0,\mathfrak{l}_1)} l_1
& 0 &  \dots & &0 & 0 & \dots & 0
\\[5pt]
\frac{-1}{\gcd(\mathfrak{l}_0,\mathfrak{l}_2)} l_0
& 0 
& 
\frac{1}{\gcd(\mathfrak{l}_0,\mathfrak{l}_2)} l_2 
& 0 & & 0& & &\\
-l_0 & 0 &  & l_3 & & 0 &\vdots &&\vdots
\\
\vdots &  &  & \vdots & \ddots & \vdots && &
\\
-l_0 & 0 & \dots & 0      &        & l_r & 0& \dots & 0
\end{array}
\right].
$$
\end{lemma}

\begin{proof}
The arguments are similar as for~\cite[Cor.~6.3]{ArBrHaWr}. 
The row lattice of $P_0$ is a sublattice 
of finite index of that of $P_1$
and thus there is a commutative diagram
$$ 
\xymatrix{
K_0 \ar[rr] 
\ar[rd] 
&
&
K_0/K_0^{\mathrm{tors}}
\\
&
\ZZ^{n+m}/\im(P_1^*) \ar[ur]
&
}
$$
We have to show, that $\ZZ^{n+m}/\im(P_1^*)$ is torsion free. 
Suitable elementary column operations on $P_1$
reduce the problem to showing that for the 
$r \times (r+1)$~matrix
$$
\left[
\begin{array}{cccccc}
\frac{-1}{\gcd(\mathfrak{l}_0,\mathfrak{l}_1)} \mathfrak{l}_0
& 
\frac{1}{\gcd(\mathfrak{l}_0,\mathfrak{l}_1)} \mathfrak{l}_1
& 0 &  \dots & &0 
\\[5pt]
\frac{-1}{\gcd(\mathfrak{l}_0,\mathfrak{l}_2)} \mathfrak{l}_0
& 0 
& 
\frac{1}{\gcd(\mathfrak{l}_0,\mathfrak{l}_2)} \mathfrak{l}_2 
& 0 & & 0
\\
-\mathfrak{l}_0 & 0 &  & \mathfrak{l}_3 & & 0 
\\
\vdots &  &  & \vdots & \ddots & \vdots
\\
-\mathfrak{l}_0 & 0 & \dots & 0  & & \mathfrak{l}_r 
\\
\end{array}
\right]
$$
the $r$-th determinantal divisor and therefore 
the product of the invariant factors equals one.
Up to sign, the $r \times r$ minors of the above 
matrix are 
$$
\frac{1}%
{\gcd(\mathfrak{l}_0,\mathfrak{l}_1)\gcd(\mathfrak{l}_0,\mathfrak{l}_2)}
\mathfrak{l}_0 \cdots \mathfrak{l}_{i-1} \cdot \mathfrak{l}_{i+1} 
\cdots \mathfrak{l}_r, 
\quad
\text{ where } i = 0, \ldots, r.
$$
Suppose that some prime $p$ divides all these
minors.
Then $p \nmid \mathfrak{l}_j$ holds for all
$j \ge 3$, because otherwise we find an $i \ne j$ with 
$p \mid \mathfrak{l}_i$, contradicting 
$\gcd$-orderedness of $P_0$.
Thus, $p$ divides each of the numbers
$$
\frac{\mathfrak{l}_0\mathfrak{l}_2}
{\gcd(\mathfrak{l}_0,\mathfrak{l}_1)\gcd(\mathfrak{l}_0,\mathfrak{l}_2)}, 
\qquad
\frac{\mathfrak{l}_1\mathfrak{l}_2}
{\gcd(\mathfrak{l}_0,\mathfrak{l}_1)\gcd(\mathfrak{l}_0,\mathfrak{l}_2)}, 
\qquad 
\frac{\mathfrak{l}_0\mathfrak{l}_1}
{\gcd(\mathfrak{l}_0,\mathfrak{l}_1)\gcd(\mathfrak{l}_0,\mathfrak{l}_2)}.
$$
By the assumption of the lemma, 
$\mathfrak{l} := \gcd(\mathfrak{l}_1, \mathfrak{l}_2)$
equals 
$\gcd(\mathfrak{l}_0,\mathfrak{l}_1, \mathfrak{l}_2)$.
Consequently, we obtain
$$
\gcd(\mathfrak{l}_0\mathfrak{l}_2,\mathfrak{l}_0\mathfrak{l}_1,\mathfrak{l}_1\mathfrak{l}_2) 
\ = \ 
\gcd(\mathfrak{l}_0\mathfrak{l}, \mathfrak{l}_1\mathfrak{l}_2)
\ = \ 
\gcd(\mathfrak{l}_0,\mathfrak{l}_1)\gcd(\mathfrak{l}_0,\mathfrak{l}_2).
$$
We conclude $p=1$; a contradiction.
Being the greatest common divisor of the above
minors, the $r$-th determinantal divisor 
equals one.
\end{proof}

\goodbreak

\begin{lemma}
\label{lem:numbercomp2}
Let $R(A,P_0)$ be as in 
Construction \ref{constr:RAP0Type2} 
and $X:=\Spec \, R(A,P_0)$ be rational.
Assume that $P_0$ is $\gcd$-ordered. 
Then the number $c(i)$ 
of irreducible components of 
$V(X,T_{ij})$ is given by

\begin{center}

\renewcommand{\arraystretch}{1.8}

\begin{tabular}{c|c|c|c|c}
$i$ & $0$ & $1$ & $2$ & $\ge 3$
\\
\hline
$c(i)$
& 
$\gcd(\mathfrak{l}_1,\mathfrak{l}_2)$
& 
$\gcd(\mathfrak{l}_0,\mathfrak{l}_2)$
& 
$\gcd(\mathfrak{l}_0,\mathfrak{l}_1)$
&
$\frac{1}{\mathfrak{l}} 
\gcd(\mathfrak{l}_1,\mathfrak{l}_2)
\gcd(\mathfrak{l}_0,\mathfrak{l}_2)
\gcd(\mathfrak{l}_0,\mathfrak{l}_1)$
\end{tabular}
\end{center}
\end{lemma}

\begin{proof}
The assertion is a direct consequence of~\cite[Lemma~6.4]{ArBrHaWr}.
\end{proof}

We are ready for the main ingredience of the 
proof of Theorem~\ref{theo:iteration}, 
the explicit description of the iterated Cox ring.

\begin{proposition}
\label{prop::isotropy}
Let $R(A,P_0)$ be non-factorial with $\Spec \, R(A,P_0)$
rational.
Assume that $P_0$ is $\gcd$-ordered and let 
$P_1$ be as in Lemma~\ref{lemma::Ptors}.
Define numbers ${n' := c(0)n_0 + \ldots + c(r)n_r}$ and
$$ 
n_{i,1}, \ldots, n_{i,c(i)} \ := \ n_i,
\qquad
l_{ij,1}, \ldots, l_{ij,c(i)}
\ := \ 
\gcd ( (P_{1})_{1,ij},\ldots,(P_{1})_{r,ij} ).
$$
Then the vectors
$l_{i,\alpha} := 
(l_{i1,\alpha}, \ldots, l_{in_i,\alpha}) \in \ZZ^{n_{i,\alpha}}$
build up an $r' \times (n' + m)$~matrix $P_0'$.
With a suitable matrix $A'$,
the affine variety $\Spec \, R(A',P'_0)$ is the 
total coordinate space of the affine variety 
$\Spec \, R(A,P_0)$.
\end{proposition}

\begin{proof}
The idea is to work with the action of the torus 
$H_0^0 := \Spec \, \CC[K_0/K_0^{\rm tors}]$ on 
$X := \Spec \, R(A,P_0)$ and to use the 
description of the Cox ring of a variety with 
torus action provided in~\cite{HaSu}.
For this, one has to look at the exceptional 
fibers of the map $\pi \colon X_0 \to Y$, 
where $X_0 \subseteq X$ is the set of points with 
at most finite $H_0^0$-isotropy and 
the curve $Y$ is the separation of $X_0 / H_0^0$.
Following the lines of the proof 
of~\cite[Prop.~6.6]{ArBrHaWr}, 
one uses Lemma~\ref{lem:numbercomp2} to
determine the number of components for each fiber
of $\pi$ and Lemma~\ref{lemma::Ptors} 
to determine the order of the general (finite)
$H_0^0$-isotropy groups on each component.
The rest is application of~\cite{HaSu}.
\end{proof}

If $R(A,P_0)$ is a hyperplatonic ring, then
$\mathfrak{l}_0^{-1} + \ldots + \mathfrak{l}_0^{-1} \ge r-1$ 
holds.
Thus, we find a (unique) platonic triple
$(\mathfrak{l}_i,\mathfrak{l}_j,\mathfrak{l}_k)$ 
with $i,j,k$ pairwise different
and all $\mathfrak{l}_u$ with~$u$ different from 
$i,j,k$ equal one.
We call $(\mathfrak{l}_i,\mathfrak{l}_k,\mathfrak{l}_k)$ 
the~\emph{basic platonic triple (bpt)} of~$R(A,P_0)$.

\begin{remark}
\label{rem:defplat}
Let $R(A,P_0)$ be non-factorial and hyperplatonic
with basic platonic triple 
$(\mathfrak{l}_0,\mathfrak{l}_1,\mathfrak{l}_2)$.
Then Remark~\ref{rem:cor} ensures that
$X := \Spec \, R(A,P_0)$ is rational.
Moreover, Lemma~\ref{lem:numbercomp2} 
and Proposition~\ref{prop::isotropy}
yield that the exponent vectors of the defining 
relations of the Cox ring $R(A',P'_0)$ of $X$ 
are computed in terms of the exponent vectors 
$l_0,\ldots, l_r$ of $R(A,P_0)$ according 
to the table below, where ``$a \times l_i$'' 
means that the vector $l_i$ shows up $a$ times:

\begin{center}

\renewcommand{\arraystretch}{1.8}

\begin{tabular}{l|l}
bpt of $R(A,P_0)$
& 
exponent vectors in $R(A',P')$
\\
\hline
$(4,3,2)$ 
& 
$2 \times l_1$, $\frac{1}{2} l_0$, $\frac{1}{2} l_2$ 
and $2 \times l_i$ for $i \geq 3$
\\
\hline
$(3,3,2)$
& 
$3 \times l_2$, $\frac{1}{3} l_0$, $\frac{1}{3} l_1$ 
and $3 \times l_i$ for $i \geq 3$
\\
\hline
$(x,2,2)$ and $2 \mid x$ 
& 
$2 \times \frac{1}{2} l_0$, $2\times \frac{1}{2}  l_1$, 
$2\times\frac{1}{2} l_2$ and $4 \times l_i$ for $i \geq 3$
\\
\hline
$(x,2,2)$ and $2 \nmid x$ 
& 
$2 \times l_0$, $\frac{1}{2} l_1$, $\frac{1}{2} l_2$
and $2 \times l_i$ for $i \geq 3$
\\
\hline
$(x,y,1)$  
& 
$\frac{1}{\gcd(x,y)} l_0$, 
$\frac{1}{\gcd(x,y)} l_1$
and $\gcd(x,y) \times l_i$ 
for $i \geq 2$
\end{tabular}

\end{center}
\end{remark}

\begin{lemma}
\label{lemma::gcd}
Let $R(A,P_0)$, arising from
Construction~\ref{constr:RAP0Type2},
be non-factorial and assume that 
$X := \Spec \, R(A,P_0)$ is rational. 
If the total coordinate space of $X$ 
is rational as well, 
then $\mathfrak{l}_{i} > 1$ holds for 
at most three $0 \leq i \leq r$.
\end{lemma}
 
\begin{proof}
We may assume that $P_0$ is $\gcd$-ordered.
Then Proposition~\ref{prop::isotropy} 
provides us with the exponent vectors of
the Cox ring $R(A',P_0')$ of $X$. 
As $R(A,P_0)$ is rational and non-factorial,
Remark~\ref{rem:cor} leaves us with the 
following two cases.

\smallskip
\noindent
\emph{Case~1.} 
We have 
$\gcd(\mathfrak{l}_0, \mathfrak{l}_1) > 1$ 
and $\gcd(\mathfrak{l}_i, \mathfrak{l}_j) = 1$ 
whenever $j \geq 2$.
This means in particular 
$\mathfrak{l}_0, \mathfrak{l}_1> 1$.
Assume that there are $2 \le i < j \le r$ 
with $\mathfrak{l}_i, \mathfrak{l}_j > 1$.
According to Proposition~\ref{prop::isotropy},
we find $c(i)$ times the exponent vector $l_i$ 
and $c(j)$ times the exponent vector $l_j$ 
in $P_0'$.
Lemma~\ref{lem:numbercomp2}
tells us $c(j) = c(i) = \gcd(\mathfrak{l}_0, \mathfrak{l}_1) > 1$.
Thus, for the first two copies of $l_i$ and $l_j$,
we obtain $\gcd(l_{i,1}, l_{i,2}) = \mathfrak{l}_{i} > 1$
and $\gcd(l_{j,1}, l_{j,2})=\mathfrak{l}_{j} > 1$
respectively. 
Remark~\ref{rem:cor} shows that 
$\Spec \, R(A', P_0')$ is not rational;
a contradiction.

\smallskip
\noindent
\emph{Case 2.} 
We have 
$\gcd(\mathfrak{l}_0, \mathfrak{l}_1)=\gcd(\mathfrak{l}_0, 
\mathfrak{l}_2)=\gcd(\mathfrak{l}_1, \mathfrak{l}_2) =2$.
Assume that there is an index 
$3 \leq i \leq r$ with $\mathfrak{l}_i > 1$. 
Proposition~\ref{prop::isotropy}
and
Lemma~\ref{lem:numbercomp2} 
yield that the exponent vector $l_i$ occurs $c(k) = 4$ 
times in the matrix $P_0'$. 
As in the previous case we conclude via 
Remark~\ref{rem:cor} that 
the total coordinate space $\Spec \, R(A', P_0')$ 
is not rational; a contradiction.
\end{proof}

\begin{proof}[Proof of Theorem~\ref{theo:iteration}]
We prove ``(ii)$\Rightarrow$(i)''.
Then $X$ is a rational and has a
hyperplatonic ring $R(A,P_0)$ 
provided by Construction~\ref{constr:RAP0Type2}
as its Cox ring.
If $R(A,P_0)$ is factorial, then there is nothing 
to show. 
So, let $R(A,P_0)$ be non-factorial.
We may assume that $P_0$ is $\gcd$-ordered.
Then  
$(\mathfrak{l}_0, \mathfrak{l}_1, \mathfrak{l}_2)$
is the basic platonic triple of $R(A,P_0)$.
From Remark~\ref{rem:defplat} we infer
that $X_1 := \Spec \, R(A,P_0)$ is rational 
with hyperplatonic Cox ring $R(A',P_0')$. 
So, we can pass to $X_2 := R(A',P_0')$
and so forth. 
The table of possible basic platonic triples
given in Remark~\ref{rem:defplat} 
shows that the iteration process terminates
at a factorial ring.

We prove ``(i)$\Rightarrow$(ii)''.
Since $X$ has a Cox ring, $X$ must have
finitely generated divisor class group.
As for any $\TT$-variety of complexity one,
the latter is equivalent to $X$ being rational.
The Cox ring of $X$ is a ring $R(A,P_0)$ as
provided by Construction~\ref{constr:RAP0Type2}.
If $R(A,P_0)$ is factorial, then we are done.
So, let $R(A,P_0)$ be non-factorial.
Then we may assume that $P_0$ is $\gcd$-ordered
and, moreover, $\mathfrak{l}_{01} \ne 1$.
Since $X_1 = \Spec \, R(A,P_0)$ has a Cox
ring $R(A',P_0')$, it must be rational.
By Lemma~\ref{lemma::gcd} we have 
$\mathfrak{l}_j = 1$ whenever $j \geq 3$ holds. 
Remark~\ref{rem:cor} leaves us with the following 
cases.

\smallskip
\noindent
\textit{Case 1.} 
We have $\mathfrak{l}_{01}:=\gcd(\mathfrak{l}_0, \mathfrak{l}_1) > 1$ 
and $\gcd(\mathfrak{l}_i, \mathfrak{l}_j) = 1$ 
whenever $j \geq 2$ holds. 
Then we may assume $\mathfrak{l}_0 \ge \mathfrak{l}_1$.

\smallskip
\noindent
\emph{1.1.} 
Consider the case $\mathfrak{l}_{01} > 3$. 
By Lemma~\ref{lem:numbercomp2}, the exponent
vector $l_2$ occurs $\mathfrak{l}_{01}$ times 
in the defining relations 
of the Cox ring $R(A',P_0')$ of $X_1$.
Since $\Spec \, R(A',P_0')$ is rational,
Remark~\ref{rem:cor} yields $\mathfrak{l}_2 =1$.
We conclude that
$(\mathfrak{l}_0, \mathfrak{l}_1, \mathfrak{l}_2)$ 
is platonic.

\smallskip
\noindent
\emph{1.2.}
Assume $\mathfrak{l}_{01} = 3$. 
Then $l_2$ occurs $3$ times as exponent vector
in the defining relations of $R(A',P_0')$. 
Remark~\ref{rem:cor} shows $\mathfrak{l}_2 \leq 2$.
Thus, 
$(\mathfrak{l}_0, \mathfrak{l}_1, \mathfrak{l}_2)$ 
is platonic.

\smallskip
\noindent
\emph{1.3.}
Let $\mathfrak{l}_{01} = 2$. 
If $\mathfrak{l}_0 = \mathfrak{l}_1 = 2$ holds,
then $(\mathfrak{l}_0, \mathfrak{l}_1, \mathfrak{l}_2)$ 
is a platonic triple for any $\mathfrak{l}_2$.
So, assume $\mathfrak{l}_0 > \mathfrak{l}_1 \geq 2$. 
As we are in Case~1, the number $\mathfrak{l}_2$ 
must be odd.
If $\mathfrak{l}_2=1$ holds, then  
$(\mathfrak{l}_0, \mathfrak{l}_1, \mathfrak{l}_2)$ 
is a platonic triple.
By Proposition~\ref{prop::isotropy}
and Lemma~\ref{lem:numbercomp2},
we find the exponent vectors 
$1/2 \,  l_0$ and $1/2 \, l_1$ 
as well as twice $l_2$ in $P_0'$. 
Since $X_1 = \Spec \, R(A',P_0')$ is rational
and $\mathfrak{l}_0 >\mathfrak{l}_1$ holds,
Lemma~\ref{lemma::gcd} shows
$\mathfrak{l}_1 = 2$ and the triple
of non-trivial gcd's of exponent vectors
of $P_0'$ is 
$(\mathfrak{l}_0/2,\mathfrak{l}_2,\mathfrak{l}_2)$.
After gcd-ordering $P_0'$, we can apply 
Case~1.1 and with $\mathfrak{l}_0/2 >1$ 
we obtain $\mathfrak{l}_0 = 4$ and 
$\mathfrak{l}_2 = 3$.
In particular, 
$(\mathfrak{l}_0, \mathfrak{l}_1, \mathfrak{l}_2)$ 
is platonic.

\smallskip
\noindent
\emph{Case 2:} 
We have 
$\gcd(\mathfrak{l}_0, \mathfrak{l}_1)
=\gcd(\mathfrak{l}_0, \mathfrak{l}_2)
=\gcd(\mathfrak{l}_1, \mathfrak{l}_2) =2$. 
Then we may assume 
$\mathfrak{l}_0 \ge \mathfrak{l}_1 \ge \mathfrak{l}_2$. 
Proposition~\ref{prop::isotropy}
and Lemma~\ref{lem:numbercomp2} tell
us that each of the exponent vectors 
$1/2 \, l_0$, $1/2 \, l_1$ and 
$1/2 \, l_2$ occurs twice in $P_0'$.
Since $\Spec \, R(A',P_0')$ is rational, 
Lemma~\ref{lemma::gcd} yields
$\mathfrak{l}_1 = \mathfrak{l}_2 = 2$.
Thus, $(\mathfrak{l}_0, \mathfrak{l}_1, \mathfrak{l}_2)$ 
is platonic.
\end{proof}

%
%

\section{Proof of Theorem~\ref{theo:iterationType1}}

As a first step we relate the total coordinate space
of a rational variety with torus action of 
complexity one admitting non-constant invariant 
functions to the total coordinate space of one with 
only constant invariant functions;
see Corollary~\ref{cor:geomembedType1}.  
This allows us to characterize rationality of the 
total coordinate space using previous results; 
see Corollary~\ref{cor:ratcharType1}.
Then we determine in a similar manner as before, the
iterated Cox ring; see Proposition~\ref{prop:iterCoxringType1}.
This finally allows us to prove Theorem~\ref{theo:iterationType1}.
We begin with recalling the necessary notions
from~\cite{HaWr}.

\begin{construction}
\label{constr:RAP0Type1}
Fix integers $r,n > 0$, $m \ge 0$ and a partition 
$n = n_1+\ldots+n_r$. For each $1 \le i \le r$, fix a tuple
$l_{i} \in \ZZ_{> 0}^{n_{i}}$ and define a monomial
$$
T_{i}^{l_{i}}
\ := \
T_{i1}^{l_{i1}} \cdots T_{in_{i}}^{l_{in_{i}}}
\ \in \
\CC[T_{ij},S_{k}; \ 1 \le i \le r, \ 1 \le j \le n_{i}, \ 1 \le k \le m].
$$
Let $A := (a_1, \ldots, a_r)$ be a list of pairwise 
different elements of~$\CC$.
Define for every $i=1, \dots, r-1$ a polynomial
$$
g_{i} 
\ := \ 
T_i^{l_i} - T_{i+1}^{l_{i+1}} - (a_{i+1}-a_i) 
\ \in \ 
\CC[T_{ij}, S_k].
$$
We build up an $r \times (n+m)$ matrix 
from the exponent vectors $l_1, \ldots, l_r$ of these 
polynomials:
$$
P_{0}
\ := \
\left[
\begin{array}{cccccc}
l_{1} &  & 0 & 0  &  \ldots & 0
\\
\vdots  & \ddots & \vdots & \vdots &  & \vdots
\\
 0 &  & l_{r} & 0  &  \ldots & 0
\end{array}
\right].
$$
Similar to the case in Construction \ref{constr:RAP0Type2}
the matrix $P_0$ defines a grading of the group
$K_0 := \ZZ^{n+m}/\im(P_0^*)$ on the ring
$$
R(A,P_0)
\ := \ 
\CC[T_{ij}, S_k]/ \bangle{g_1,\dots, g_{r-1}}.
$$
\end{construction}

Following~\cite{HaWr} we call a ring $R(A,P_0)$ 
arising from Construction~\ref{constr:RAP0Type1} 
of \emph{Type~1}
and a ring $R(A,P_0)$ as in 
Construction~\ref{constr:RAP0Type2} of \emph{Type~2}.
According to~\cite{HaWr}, the suitable downgradings 
of the rings $R(A,P_0)$ of Type~1 yield precisely the 
Cox rings of the normal rational $\TT$-varieties $X$ 
of complexity one with $\Gamma(X,\mathcal{O})^{\TT} = \CC[T]$.

\begin{construction}
\label{constr:embedtype1}
Consider a ring $R(A,P_0)$ of Type~1.
Set $\mathfrak{l}_i := \gcd(l_{i1}, \ldots, l_{in_i})$
and  
$\ell := \lcm(\mathfrak{l}_1, \ldots, \mathfrak{l}_r)$.
Then, writing $L_0$ for the column vector
$-(\ell, \ldots, \ell) \in \ZZ^r$,
we obtain a ring $R(\tilde A, \tilde P_0)$ of 
Type~2 with defining matrices 
$$ 
\tilde A
\ := \
\left[
\begin{array}{rrrr}
-1 & a_1 & \dots & a_r
\\
0 & 1 & \dots & 1
\end{array}
\right],
\qquad
\tilde P_0 
\ := \ 
\left[L_0,P_0\right].
$$
\end{construction}

\begin{proposition}
\label{prop:algembType1}
Let $R(A,P_0)$ be a ring of Type~1 and 
$R(\tilde A, \tilde P_0)$ the associated
ring of Type~2 obtained via 
Construction~\ref{constr:embedtype1}.
Fix $\alpha_{ij} \in \ZZ$ with 
$\mathfrak{l}_i = \alpha_{i1} l_{i1} + \ldots + \alpha_{in_i} l_{in_i}$.
Then one obtains an isomorphism of graded
$\CC$-algebras
$$ 
R(\tilde A, \tilde P_0)_{\tilde T_{01}}
\ \to \ 
R(A,P_0)[T_{01},T_{01}]^{-1},
\qquad
\tilde T_{01} \ \mapsto \ T_{01},
\quad
\tilde T_{ij} \ \mapsto \ T_{ij} T_{01}^{\frac{\ell}{\mathfrak{l}_i} \alpha_{ij}}.
$$
\end{proposition}

\begin{proof}
By construction, $R(\tilde A, \tilde P_0)$ is a factor 
algebra of $\CC[\tilde T_{ij}, \tilde S_k]$ and 
$R(A,P_0)$ of $\CC[T_{ij}, S_k]$. 
We have an isomorphism of $\CC$-algebras
$$
\psi 
\colon 
\CC[\tilde T_{ij}, \tilde S_k]_{\tilde T_{01}}
 \to  
\CC[T_{ij}, S_k][T_{01},T_{01}]^{-1},
\quad
\tilde T_{01} \mapsto T_{01},
\
\tilde T_{ij} \mapsto T_{ij} T_{01}^{\frac{\ell}{\mathfrak{l}_i} \alpha_{ij}},
\
\tilde S_k \mapsto S_k.
$$ 
Observe $\psi(\tilde T_i^{l_i}) = T_{01}^{\ell} T_i^{l_i}$. 
We claim that $\psi$ is compatible
with the gradings by $\tilde K_0$ on the l.h.s. 
and by $\ZZ \times K_0$ on the r.h.s.,
where the latter grading is given by 
$$
\deg(T_{01}) \ = \ (1,0) \ \in \ \ZZ \times K_0,
\qquad
\deg(T_{ij}) \ = \ (0,e_{ij} + {\rm im}(P_0^*))  \ \in \ \ZZ \times K_0.
$$ 
Indeed, because of 
$\psi(\tilde T_{01}^{-\ell} \tilde T_i^{l_i}) = T_i^{l_i}$,
the kernels of the respective downgrading maps
$$ 
\ZZ^{n+1+m} \ \to \ \tilde K_0,
\qquad\qquad
\ZZ^{n+1+m} \ \to \ \ZZ \times K_0,
$$
generated by the rows  $\tilde P_0$ and $P_0$, 
correspond to each other under $\psi$.
The defining ideal of $R(\tilde A, \tilde P_0)$
is generated by the polynomials 
$\tilde g_1, \ldots, \tilde g_{r-1}$, where
$$
\tilde{g}_{i}
\ :=  \
\det
\left[
\begin{array}{ccc}
\tilde T_0^{\ell} & T_{i}^{l_{i}} & T_{i+1}^{l_{i+1}}
\\
-1 & a_{i}& a_{i+1}
\\
0 & 1 & 1
\end{array}
\right].
$$
The above isomorphism sends $\tilde g_i$ to
$T_0^{\ell} g_i$, where the $g_i$ are the 
generators of the defining ideal of $R(A,P_0)$,
and thus induces the desired isomorphism.
\end{proof}

\begin{corollary}
\label{cor:geomembedType1}
Let $X := \Spec \, R(A,P_0)$ be the affine variety 
arising from a ring of Type~1 and 
$\tilde X := \Spec \, R(\tilde A, \tilde P_0)$
the one arising from the associated ring of Type~2.
Then $X \times \CC^*$ is isomorphic to the principal
open subset $\tilde X_{\tilde T_{01}} \subseteq \tilde X$. 
In particular, $X$ is rational if and only if $\tilde X$ 
is so.
\end{corollary}

\begin{corollary}
\label{cor:ratcharType1}
Let $R(A,P_0)$ be a ring of Type~1. 
Then $X =\Spec \, R(A,P_0)$ is rational 
if and only if one of the following 
conditions holds:
\begin{enumerate}
\item 
One has $\mathfrak{l}_i = 1$ for all $1 \le i \le r$, 
in other words, $R(A,P_0)$ is factorial.
\item 
There is exactly one $1 \le i \le r$ with $\mathfrak{l}_i >1$.
\item 
There are $1 \leq i < j \leq r$ with $\mathfrak{l}_i = \mathfrak{l}_j =2$ 
and $\mathfrak{l}_u=1$ whenever $u \notin \left\{i, j\right\}$
\end{enumerate}
\end{corollary}

\begin{proof}
Combine Corollary~\ref{cor:geomembedType1} with the 
rationality criterion Remark~\ref{rem:cor}.
\end{proof}

\begin{lemma}
\label{lem:irrcomptype1}
Let $R(A,P_0)$ be of Type~1 with $X:=\Spec \, R(A,P_0)$ 
rational and assume that $(\mathfrak{l}_1, \dots, \mathfrak{l}_r)$ 
is decreasingly ordered. 
Then the number $c(i)$ of irreducible components of 
$V(X, T_{ij})$ is given as
\begin{center}

\renewcommand{\arraystretch}{1.8} 

\begin{tabular}{c|c|c|c}
$i$ & $1$ & $2$ & $\ge 3$
\\
\hline
$c(i)$
& 
$\mathfrak{l}_1$ 
& 
$\mathfrak{l}_2$
& 
$\mathfrak{l}_1 \mathfrak{l}_2$
\end{tabular}
\end{center}
\end{lemma}

\begin{proof}
Due to Corollary~\ref{cor:geomembedType1},
we can realize $X \times \CC^*$ as a principal
open subset of the associated variety $\tilde{X}$ 
of Type~2. 
Then the irreducible components of 
$V(X,T_{ij}) \times \CC^*$ are 
in one-to-one correspondence with 
the irreducible components 
$X \cap V(\tilde{X} , \tilde T_{ij})$.
The assertions follows.
\end{proof}

\begin{proposition}
\label{prop:iterCoxringType1}
Let $R(A,P_0)$ be non-factorial of Type~1 with 
$\Spec \, R(A,P_0)$ rational and 
$(\mathfrak{l}_1, \ldots, \mathfrak{l}_r)$
decreasingly ordered. 
Define numbers ${n' := c(1)n_1 + \ldots + c(r)n_r}$
and
$$ 
n_{i,1}, \ldots, n_{i,c(i)} \ := \ n_i,
\qquad\qquad
l_{i,1}, \ldots, l_{i,c(i)}
\ := \ 
\frac{1}{\mathfrak{l}_i} l_i.
$$
Then the vectors
$l_{i,\alpha} \in \ZZ^{n_{i,\alpha}}$
build up an $r' \times (n' + m) $~matrix $P_0'$.
With a suitable matrix $A'$ 
the affine variety $\Spec \, R(A',P'_0)$ is the 
total coordinate space of the affine variety 
$\Spec \, R(A,P_0)$.
\end{proposition}

\begin{proof}
First observe that the kernel of 
$\ZZ^{n+m} \rightarrow K_0/K_0^{\mathrm{tors}}$ 
is generated by the rows of the following 
$r \times (n+m)$~matrix:
$$
\left[
\begin{array}{cccccc}
\frac{1}{\mathfrak{l}_1} l_1
& 
\dots
& 
0
&
0 & \dots & 0
\\
\vdots
& 
\ddots 
& 
\vdots
&
\vdots &  & \vdots
\\
0 
& 
\dots
&
\frac{1}{\mathfrak{l}_r} l_r
&
0 & \dots & 0
\end{array}
\right].
$$
Now one determines the Cox ring of $X = \Spec \,  R(A,P_0)$ 
in the same manner as in the proof of~\cite[Prop.~6.6]{ArBrHaWr} 
by exchanging the matrix $P_1$ used there by the matrix above
and applying Lemma~\ref{lem:irrcomptype1}.
\end{proof}

\begin{proof}[Proof of Theorem~\ref{theo:iterationType1}]
If $R(A,P_0)$ is rational of Type~1, 
then Proposition~\ref{prop:iterCoxringType1} shows that 
the Cox ring of $\Spec \, R(A,P_0)$ is factorial.
Thus, Cox ring iteration is possible for $X$ if and only if 
the total coordinate space of $X$ is rational.
Moreover, if the latter holds then the Cox ring iteration 
ends with at most one step.
\end{proof}


\end{document}